\documentclass[12pt,draft]{scrartcl}

\usepackage[]{amsmath, amssymb,amsfonts,amsthm,mathtools}
\usepackage{braket}
\usepackage{enumitem}
\usepackage{comment}
\usepackage{color}
\usepackage{multicol}

\makeatletter
\DeclareOldFontCommand{\rm}{\normalfont\rmfamily}{\mathrm}
\DeclareOldFontCommand{\sf}{\normalfont\sffamily}{\mathsf}
\DeclareOldFontCommand{\tt}{\normalfont\ttfamily}{\mathtt}
\DeclareOldFontCommand{\bf}{\normalfont\bfseries}{\mathbf}
\DeclareOldFontCommand{\it}{\normalfont\itshape}{\mathit}
\DeclareOldFontCommand{\sl}{\normalfont\slshape}{\@nomath\sl}
\DeclareOldFontCommand{\sc}{\normalfont\scshape}{\@nomath\sc}
\makeatother

\theoremstyle{plain}% default
\newtheorem{theorem}{Theorem}[section]
\newtheorem{lemma}[theorem]{Lemma}
\newtheorem{proposition}[theorem]{Proposition}

\theoremstyle{definition}
\newtheorem{definition}[theorem]{Definition}
\newtheorem{conjecture}[theorem]{Conjecture}

\title {Monomial generalized almost perfect nonlinear functions} 
\author{
Masamichi Kuroda
}
\date{}
\begin{document}
\maketitle
	
\begin{abstract}
Generalized almost perfect nonlinear (GAPN) functions were defined to satisfy some generalizations of basic properties of almost perfect nonlinear (APN) functions for even characteristic. 
In this paper, we study monomial GAPN functions for odd characteristic. 
In particular, we give all monomial GAPN functions whose algebraic degree are maximum or minimum on a finite field of odd characteristic. 
\end{abstract}



\section{Introduction}

Let $F = \mathbb{F}_{p^n}$ be a finite field of characteristic $p$. 
A function $f \colon F \to F$ is an almost perfect nonlinear (APN) function if 
\begin{align*}
N_f (a, b) \coloneqq \# \Set{x \in F | 
D_{a}f(x) \coloneqq f(x+a) - f (x) =b} \leq 2
\end{align*}
for all $a \in F^\times$ and $b \in F$. 
When $p=2$, such functions have useful properties and applications 
in cryptography, finite geometries and so on. 
On the other hand, APN functions for odd characteristic have quite different properties from 
the even characteristic case. 
In \cite{GAPN1}, 
the definition of APN functions for odd characteristic was modified to satisfy some similar properties of APN functions for even characteristic. 
In fact, the author defined a generalized almost perfect nonlinear (GAPN) functions as follows 
(see \cite[Definition 1.1]{GAPN1}): 
A function $f \colon F \to F$ is a GAPN function if 
\begin{align*}
\tilde{N}_f (a, b) \coloneqq \# \Set{x \in F | 
\tilde{D}_{a}f(x) \coloneqq \sum_{i \in \mathbb{F}_{p}}f(x+ia) =b} \leq p
\end{align*}
for all $a \in F^\times$ and $b \in F$. 
Note that when $p=2$, GAPN functions coincide with APN functions. 
%In \cite{GAPN1}, the author proved that  GAPN functions for odd characteristic satisfy some properties which are generalizations of basic properties of APN functions for even characteristic. 
In addition, a few examples of GAPN functions on $F = \mathbb{F}_{p^n}$ was constructed. 
For example, 
\begin{itemize}
\item the inverse permutation 
\begin{align*}
f \colon F \longrightarrow F, \ \ x \longmapsto x^{p^n - 2}
\end{align*}
is a GAPN function of algebraic degree $n (p-1) - 1$ if $p$ is odd, and 
\item the generalized Gold function 
\begin{align*}
f \colon F \longrightarrow F, \ \ x \longmapsto x^{p^i + p - 1}. 
\end{align*}
is a GAPN function of algebraic degree $p$ if $\gcd (i,n) = 1$. 
%We call them the generalized Gold functions (see Table \ref{monomial APN} for the Gold function). 
\end{itemize}
Here see \cite[Section 2]{GAPN1} for the algebraic degree, and see Table 1 below for the Gold functions. 

In this paper, we study monomial GAPN functions for odd characteristic. 
Note that monomial APN functions for even characteristic have been studied by many researchers. 
The following Table \ref{monomial APN} is a complete list, up to CCZ-equivalence, of known monomial APN functions for even characteristic, 
where $d^{\circ} (f_d)$ is the algebraic degree of $f_d$. 
We will give a generalization of Welch functions (see sub-subsection \ref{generalized Welch}). 
\begin{table}[h]
\begin{center}
\caption{Known monomial APN functions $f_d (x) = x^d$ on $\mathbb{F}_{2^n}$} \label{monomial APN}
  \begin{tabular}{|l|c|c|c|c|} \hline
& Exponents $d$ & Conditions & $d^{\circ} (f_d)$ & References \\ \hline \hline 
Gold functions & $2^i+1$ & $\gcd (i,n) = 1$ & $2$ & 
\cite{Gold1968maximal} \cite{Nyberg1994differentially} \\ \hline
Kasami functions & $2^{2i} - 2^i+1$ & $\gcd (i,n) = 1$ & $i+1$ & 
\cite{JW93} \cite{Kasami71} \\ \hline
Welch functions & $2^t+3$ & $n = 2t + 1$ & $3$ & \cite{Dobbertin99Welch} \\ \hline
Niho functions & $2^t + 2^{\frac{t}{2}}-1$, $t$ is even & $n=2t+1$ & $\frac{t+2}{2}$ & 
\cite{Dobbertin99Niho} \\ 
                  & $2^t + 2^{\frac{3t+1}{2}}-1$, $t$ is odd & $n=2t+1$ & $t+1$ &  \\ \hline
Inverse function & $2^{2t}-1$ & $n=2t+1$ & $n-1$ & 
\cite{BD94} \cite{Nyberg1994differentially} \\ \hline
Dobbertin functions & $2^{4t}+2^{3t}+2^{2t}+2^t - 1$ & $n=5t$ & $t+3$ & \cite{Dobbertin01} \\ \hline
  \end{tabular}
\end{center}
\end{table}

Every function $f \colon F \to F$ can be represented uniquely as a polynomial function 
$f(x) = \sum_{d=0}^{p^n - 1} c_d x^d \in F [x]$. 
We can extend $f$ to an extension field of $F$ by using this unique polynomial formula. 
Then we can define a generalization of exceptional APN functions as follows 
(see \cite{AMR10} and \cite{HM2011} for exceptional APN functions): 
\begin{definition}
\begin{enumerate}
\item[(1)] A function $f \colon F \to F$ is \textbf{$p$-exceptional} 
if $f$ is a GAPN function on $F$ and 
is also GAPN function on infinitely many extension fields of $F$. 
\item[(2)] The exponent $d$ is 
$p$-\textbf{exceptional} if $f(x) = x^d$ is a GAPN function 
on infinitely many extension fields of $\mathbb{F}_{p}$. 
\end{enumerate}
\end{definition}
Note that $2$-exceptional exponents are so-called exceptional exponents. 
For any $p$, generalized Gold functions are $p$-exceptional clearly. 
When $p=2$, Kasami functions are also $2$-exceptional functions. 
In addition, the following Theorem was conjectured by Dillon \cite{Dillon02} and was proved by 
Hernando and McGuire \cite{HM2011}: 

\begin{theorem}
When $p=2$, the only $2$-exceptional monomial APN functions are the Gold and Kasami functions. In other words, the only $2$-exponential exponents are the Gold and Kasami numbers. 
\end{theorem}

%We will give partial classifications of monomial GAPN functions. 
%In particular, w
In this paper, we will give all monomial GAPN functions of algebraic degree 
$p$ or $n (p-1) - 1$ (see Subsection \ref{minimum algebraic degree} and \ref{maximum algebraic degree}, respectively). 
Note that %we can check easily that 
if $f \colon \mathbb{F}_{p^n} \to \mathbb{F}_{p^n}$ is a GAPN function, then 
we have $p \leq d^{\circ} (f) \leq n (p-1) - 1$ (see Proposition \ref{Prop: basic result}). 
In addition, we will show that when $p \geq 3$, 
any monomial function of algebraic degree $p$ is $p$-exceptional on 
some extension field of $\mathbb{F}_p$ (see Proposition \ref{main1}). 
Moreover for odd prime $p$, 
we will give a conjecture for the existences of 
GAPN functions $f$ on 
$\mathbb{F}_{p^n}$ with 
$p < d^{\circ} (f) < n (p-1) - 1$, 
and $p$-exceptional exponents 
(see Conjecture \ref{conj 1}). 

\section{Monomial GAPN functions}

In this section, we mainly assume that $p$ is an odd prime. 
Let $F = \mathbb{F}_{p^n}$ be a finite field of characteristic $p$.  
Let $f_d$ be a monomial function 
\begin{align*}
f_d \colon F \longrightarrow F, \ \ x \longmapsto x^d. 
\end{align*}
For any $x \in F$, we have that $x^{p^n} = x$, so we may assume that 
$d < p^n$. 
Then the exponent $d$ has the $p$-adic expansion 
$d = \sum_{s=0}^{n-1} d_s p^s$. 
Let $w_p (d)$ denote the sum of the coefficients 
$d = \sum_{s=0}^{n-1} d_s$, and we call it the $p$-weight of $d$. 
Clearly, we have $w_p (d) \leq n (p-1)$. 
For any monomial function $f_d \colon F \to F$, the algebraic degree $d^{\circ} (f_d)$ of $f_d$ coincides with the $p$-weight of $d$ (see \cite[Section 2]{GAPN1} for more details): 
\begin{align*}
d^{\circ} (f_d) = w_p (d) \leq n (p-1). 
\end{align*}

\begin{proposition} \label{Prop: basic result}
On the above notations, 
\begin{enumerate}
\item[(1)] if $d^{\circ} (f_d) < p$, then $f_d$ is not a GAPN function on $F$, and 
\item[(2)] if $d^{\circ} (f_d)$ is even, then $f_d$ is not a GAPN function on $F$. 
\end{enumerate}
In particular, for any monomial GAPN function $f_d$, we have 
$p \leq d^{\circ} (f_d) \leq n (p-1) - 1$. 
\end{proposition}

\begin{proof}
(1) is clear from Proposition 2.12 in \cite{GAPN1}. 
We prove (2). 
Assume that $d^{\circ} (f_d)$ is even. 
Then $d$ is even. 
Hence if the equation $\tilde{D}_a f_d (x) = b$ has a solution $x_0$, then 
any point in $\left( x_0 + a \mathbb{F}_p \right) \cup \left( - x_0 + a \mathbb{F}_p \right)$ 
is also a solution of the equation. 
Hence $\tilde{N}_{f_d} (a, b) \geq 2p$, and hence $f_d$ is not a GAPN function. 
\end{proof}

In the following, we will give all GAPN functions on $F = \mathbb{F}_{p^n}$ 
with algebraic degree $p$ or $n (p-1) - 1$.

\subsection{Monomial GAPN functions on $\mathbb{F}_{p^n}$ with minimum algebraic degree}
\label{minimum algebraic degree}

%In this subsection, let $p$ be an add prime and 
Let $f_d$ be a monomial function on $F = \mathbb{F}_{p^n}$ with $d^{\circ} (f_d) = p$. 
Then we have that 
\begin{align*} 
d = p^{i_1} + p^{i_2} + \cdots + p^{i_p} \ \ \mbox{with} \ \ i_1 \leq i_2 \leq \cdots \leq i_p. 
%\ \  \ 
%(i_2, \dots , i_p \geq 0, \ (i_2, \dots, i_p) \ne (0, \dots, 0)). 
\end{align*}
Hence we have that 
\begin{align*}
d = p^{i_i} d', \ \ \mbox{where} \ \ d' = 1 + p^{i_2 - i_1} + \cdots + p^{i_p - i_1}, 
\end{align*}
and hence we obtain $f_d = f_{d'} \circ {\rm Fb}_{p^{i_1}}$, where ${\rm Fb}_{p^{i_1}}$ is the Frobenius  isomorphism ${\rm Fb}_{p^{i_1}} (x) = x^{p^{i_1}}$. 
In particular, $f_d$ and $f_{d'}$ are EA-equivalence (see \cite[Section 2]{GAPN1}), and hence 
$f_d$ is a GAPN function if and only if $f_{d'}$ is a GAPN function by Proposition 2.1 in \cite{GAPN1}. 
Therefore we may assume that $i_1 = 0$, that is, we may assume that 
\begin{align}\label{p-exceptional}
d = 1 + p^{i_2} + \cdots + p^{i_p} 
\ \ \mbox{with} \ \ 
0 \leq i_2 \leq \cdots \leq i_p \ \ \mbox{and} \ \ (i_2, \dots, i_p) \ne (0, \dots, 0). 
\end{align}
Then we can write 
\begin{align} \label{alpha}
d = \sum_{s=0}^{n-1} \alpha_s p^s \ \ \ 
(\alpha_s \in \mathbb{F}_p, \ \alpha_0 + \cdots + \alpha_{n-1} = p). 
\end{align}
We define the polynomial $D(X) \in \mathbb{F}_p [X]$ as follows: 
\begin{align} \label{polynomial D}
D(X) \coloneqq \sum_{s=0}^{n-1} \alpha _s X^s 
\left( = 1 + X^{i_2} + \cdots + X^{i_p} \right) \in \mathbb{F}_{p} [X]. 
\end{align}
On the above notations, we have the following criterion: 
\begin{theorem} \label{Thm: criterion}
%Let $f_d$ be a monomial function on $F$ defined by 
%\begin{align*}
%f_d (x) = x^{1 + p^{i_2} + \cdots + p^{i_p}} \ \  \ 
%(i_2, \dots , i_p \geq 0, \ (i_2, \dots, i_p) \ne (0, \dots, 0)). 
%\end{align*}
$f_d$ is a GAPN function on $F$ 
if and only if 
$D (\beta) \ne 0$ 
for any $\beta \in \overline{\mathbb{F}_p} \setminus \{ 1\}$ such that $\beta^n = 1$, 
%\begin{align*}
%\Set{ \beta \in \overline{\mathbb{F}_p} | D(\beta) = 0 } 
%\cap 
%\Set{ \gamma \in \overline{\mathbb{F}_p} | \gamma^n = 1 } 
%= \left\{ 1 \right\}, 
%\end{align*}
where 
$\overline{\mathbb{F}_p}$ is the algebraic closure of $\mathbb{F}_p$. 
\end{theorem}

\begin{proof}
%By the proof of Lemme 3.3 in [GAPN1], we obtain that 
%\begin{align*}
%\tilde{D}_a {f_d} (0) &= 0, \\
%\tilde{D}_a {f_d} (x) &= - \left( a^{d-1} x + a^{d - p^{i_2}} x^{p^{i_2}} + \cdots + 
%a^{d - p^{i_p}} x^{i_p} \right)
%\end{align*}
%Moreover by the proof of Proposition 2.11 in [GAPN1], 
%if the equation $\tilde{D}_a f_d (x) = b$ has a solution, then
%\begin{align*}
%\tilde{N}_{f_d} (a, b) = 
%\tilde{N}_{f_d} \left( a, \tilde{D}_a f_d (0) \right) =
%\tilde{N}_{f_d} \left( a, 0 \right) = 
%\tilde{N}_{f_d} \left( 1, 0 \right). 
%\end{align*}
%On the other hand, the equation 
%\begin{align*}
%0 = \tilde{D}_1 f_d (x) = - \left( x + x^{p^{i_2}} + \cdots + x^{i_p} \right)
%\end{align*}
%has trivial $p$ solutions $0$, $1$, $\dots$, $p-1$. 
%Therefore we obtain the following Lemma: 
%\begin{lemma}
%$f_d$ is a GAPN function on $F$ 
%if and only if 
%$$ \Set{ x \in F | x + x^{p^{i_2}} + \cdots + x^{p^{i_{p}}} = 0 } 
%= \mathbb{F}_p .$$ 
%\end{lemma}
By the proof of Lemme 3.3 in \cite{GAPN1}, we obtain that 
\begin{align*}
\varphi_d (x) \coloneqq \tilde{D}_1 f_d (x) = x + x^{p^{i_2}} + \cdots + x^{p^{i_{p}}} \ \ (x \in F), 
\end{align*}
and hence $\varphi_d \colon F \to F$ is $\mathbb{F}_p$-linear. 
Thus $\dim_{\mathbb{F}_p} {\rm Ker} \left( \varphi_d \right) 
= n - \dim_{\mathbb{F}_p} {\rm Im} \left( \varphi_d \right) $. 
Then by Lemma 3.3, (iii) in \cite{GAPN1}, 
$f_d$ is a GAPN function on $F$ 
if and only if 
\begin{align*} %\label{Lem 3.3, (iii)}
 {\rm Ker} \left( \varphi_d \right) =  \Set{ x \in F | x + x^{p^{i_2}} + \cdots + x^{p^{i_{p}}} = 0 } 
= \mathbb{F}_p .
\end{align*} 
%On the other hand, the function 
%\begin{align*}
%\varphi_d := \tilde{D}_1 f_d \colon F \longrightarrow F , \ \ 
%x \longmapsto x + x^{p^{i_2}} + \cdots + x^{p^{i_{p}}}
%\end{align*}
%is $\mathbb{F}_p$-linear, and we obtain 
%\begin{align*}
%\dim_{\mathbb{F}_p} {\rm Ker} \left( \varphi_d \right) 
%= n - \dim_{\mathbb{F}_p} {\rm Im} \left( \varphi_d \right) . 
%\end{align*}
%Hence (\ref{Lem 3.3, (iii)}) 
This is equivalent to that 
\begin{align} \label{dim im}
\dim_{\mathbb{F}_p} {\rm Im} \left( \varphi_d \right) = n  - 1, \ 
\mbox{that is}, \ \# {\rm Im} \left( \varphi _d \right) = p^{n-1} . 
\end{align}
On the other hand, by Theorem 2.5 in \cite{CCMPT2013}, we have that 
$\# {\rm Im} \left( \varphi _d \right) = p^{{\rm rk} (M_d)}$. Here $M_d$ is the 
$n \times n $ matrix defined by 
\begin{align*}
M_d = \left[
\begin{array}{cccc}
\alpha_0 & \alpha_{n-1} & \cdots & \alpha_1 \\
\alpha_1 & \alpha_0 & \cdots & \alpha_2 \\
\vdots & \vdots & \ddots & \vdots \\ 
\alpha_{n-1} & \alpha_{n-2} & \cdots& \alpha_0 \\
\end{array}
\right], 
\end{align*}
where $\alpha _0$, $\dots$, $\alpha_{n-1}$ are defined by (\ref{alpha}). 
%and $\alpha_j$ ($\in \mathbb{F}_{p}$) is the coefficient of $x^{p^{j}}$ in 
%$x + x^{p^{i_2}} + \cdots + x^{p^{i_{p}}}$. 
Hence 
(\ref{dim im}) is equivalent to that ${\rm rk} (M_d) = n - 1$. 
Let $I_n$ be the identity matrix of size $n$. 
Then we have that 
\begin{align*}
\det \left( \mu I_n - M_d \right) 
&= \det \left[
\begin{array}{cccc}
\mu - \alpha_0 & - \alpha_{n-1} & \cdots & - \alpha_1 \\
-\alpha_1 & \mu - \alpha_0 & \cdots & -\alpha_2 \\
\vdots & \vdots & \ddots & \vdots \\ 
-\alpha_{n-1} & -\alpha_{n-2} & \cdots& \mu - \alpha_0 \\
\end{array}
\right] \\ 
&= \prod_{\beta^n=1} 
\left( \mu - \alpha_0 - \alpha_1 \beta - \dots - \alpha_{n-1} \beta ^{n-1} \right)
= \mu \prod_{\beta^n=1, \beta \ne 1}  \left( \mu - D (\beta) \right), 
\end{align*}
and hence 
the eigenvalues of $M_d$ are $0$ and 
$D (\beta)$ ($\beta \in \overline{\mathbb{F}_p} \setminus \{ 1\}$, $\beta^n = 1$). 
Therefore ${\rm rk} (M_d) = n - 1$ if and only if 
$D (\beta) \ne 0$ 
for any $\beta \in \overline{\mathbb{F}_p} \setminus \{ 1\}$ with $\beta^n = 1$. 
\end{proof}

Generalized Gold functions 
$f \colon \mathbb{F}_{p^n} \to \mathbb{F}_{p^n}$, $f(x) = x^{p^i + p-1}$ are GAPN functions if $\gcd (i, n) =1$. In particular, they are $p$-exceptional clearly. 
More generally, we obtain the following Proposition: %, which is a generalization of 

%\begin{corollary}
%The exponent $d$ given by (\ref{p-exceptional}) is $p$-exceptional exponent. 
%\end{corollary}

\begin{proposition} \label{main1}
%\begin{enumerate}
%\item[(i)] 
%For any exponent $d$ given by (\ref{p-exceptional}), 
%there exists $n \in \mathbb{N}$ with $d < p^n$ such that 
%$f_d \colon \mathbb{F}_{p^n} \to \mathbb{F}_{p^n}$ is a monomial GAPN function of algebraic degree $p$. 
%\item[(ii)] 
%Any monomial GAPN function of algebraic degree $p$ on $\mathbb{F}_{p^n}$ for some $n$ 
%is $p$-exceptional. 
%\end{enumerate}
%In particular, a
Any exponent $d$ given by (\ref{p-exceptional}) is a $p$-exceptional exponent. 
\end{proposition}

\begin{proof}
It follows immediately from the following Lemma \ref{key lemma}. 
\end{proof}

\begin{lemma} \label{key lemma}
\begin{enumerate}
\item[(i)] 
For any exponent $d$ given by (\ref{p-exceptional}), 
there exists $n \in \mathbb{N}$ with $d < p^n$ such that 
$f_d$ is a monomial GAPN function of algebraic degree $p$ on $\mathbb{F}_{p^n}$. 
\item[(ii)] 
Any monomial GAPN function of algebraic degree $p$ on $\mathbb{F}_{p^n}$ for some $n$ 
is $p$-exceptional. 
\end{enumerate}
%In particular, any exponent $d$ given by (\ref{p-exceptional}) is a $p$-exceptional exponent. 
\end{lemma}

\begin{proof}
We first prove (i). 
Let $d$ be an exponent given by (\ref{p-exceptional}). 
By Theorem \ref{Thm: criterion}, it is sufficient to show that 
there exists $n \in \mathbb{N}$ with $d < p^n$ such that 
\begin{align} \label{want (i)}
\Set{ \beta \in \overline{\mathbb{F}_{p}} | D(\beta) = 0 } 
\cap \Set{ \gamma \in \overline{\mathbb{F}_{p}} | \gamma ^n = 1} = \{ 1 \}, 
\end{align}
where $D(X) \in \mathbb{F}_{p} [X]$ is defined by (\ref{polynomial D}). 
Note that the polynomial $D(X)$ depends only on exponent $d$. Then the set 
$\Set{ \beta \in \overline{\mathbb{F}_{p}} | D(\beta) = 0 }$ is a finite set. 
Let $\beta_1$, $\dots$, $\beta_m \in 
\Set{ \beta \in \overline{\mathbb{F}_{p}} | D(\beta) = 0 } \setminus \{ 1 \}$ 
be all elements which have finite orders and let 
\begin{align*}
n_j \coloneqq \min \Set{ N \in \mathbb{N} | \beta_j^N = 1 } (> 1). 
\end{align*}
Then there exist $n \in \mathbb{N}$ with $d < p^n$ such that 
$
\gcd (n, n_j) = 1$ ($j = 1$, $\dots$, $m$). 
Then we have $\beta_j \not \in \Set{ \gamma \in \overline{\mathbb{F}_{p}} | \gamma ^n = 1}$ 
($j=1$, $\dots$, $m$) and hence we obtain (\ref{want (i)}). 

Next we prove (ii). 
Let $f_d$ be a monomial GAPN function of algebraic degree $p$ on $\mathbb{F}_{p^n}$ for some $n$. By theorem \ref{Thm: criterion}, we have that 
\begin{align} \label{intersection}
\Set{ \beta \in \overline{\mathbb{F}_{p}} | D(\beta) = 0 } 
\cap \Set{ \gamma \in \overline{\mathbb{F}_{p}} | \gamma ^n = 1} = \{ 1 \}. 
\end{align}
Similarly to above, let $\beta_1$, $\dots$, $\beta_m \in 
\Set{ \beta \in \overline{\mathbb{F}_{p}} | D(\beta) = 0 } \setminus \{ 1 \}$ 
be all elements which have finite orders and let $n_j \coloneqq \min \Set{ N \in \mathbb{N} | \beta_j^N = 1 }$ ($> 1$). 
By (\ref{intersection}), 
we obtain that $\beta_j \not \in \Set{ \gamma \in \overline{\mathbb{F}_{p}} | \gamma ^n = 1}$, 
that is, $n$ is not divisible by 
$n_j$ for each $j=1$, $\dots$, $m$. 
Then for any prime $q$ such that $\gcd (n_j , q) = 1$ ($j=1$, $\dots$, $m$), 
the number $q n$ is not divisible by $n_j$ ($j=1$, $\dots$, $m$), that is, 
$
\beta_j \not \in \Set{ \gamma \in \overline{\mathbb{F}_{p}} | \gamma ^{qn} = 1}
$ ($j=1$, $\dots$, $m$). 
Hence we get 
\begin{align*}
\Set{ \beta \in \overline{\mathbb{F}_{p}} | D(\beta) = 0 } 
\cap \Set{ \gamma \in \overline{\mathbb{F}_{p}} | \gamma ^{qn} = 1} = \{ 1 \}, 
\end{align*}
and hence $f_d$ is also GAPN function on $\mathbb{F}_{p^{qn}}$, which is an extension field of $\mathbb{F}_{p^n}$. 
Since there exist infinitely many such prime numbers, $f_d$ is $p$-exceptional. 
\end{proof}

%\begin{lemma} \label{eigenvalues of Md}
%The eigenvalues of $M_d$ are 
%$
%\lambda (\beta)$, where $\beta \in \overline{\mathbb{F}_p}$ with $\beta^n = 1$. 
%Note that $\lambda (1) = 0$. 
%\end{lemma}
%
%\begin{proof}
%Let $I_n$ be the identity matrix of size $n$. 
%We have that 
%\begin{align*}
%\det \left( \lambda I_n - M_d \right) 
%&= \det \left[
%\begin{array}{cccc}
%\lambda - \alpha_0 & - \alpha_{n-1} & \cdots & - \alpha_1 \\
%-\alpha_1 & \lambda - \alpha_0 & \cdots & -\alpha_2 \\
%\vdots & \vdots & \ddots & \vdots \\ 
%-\alpha_{n-1} & -\alpha_{n-2} & \cdots& \lambda - \alpha_0 \\
%\end{array}
%\right] \\ 
%&= \prod_{\beta^n=1} 
%\left( \lambda - \alpha_0 - \alpha_1 \beta - \dots - \alpha_{n-1} \beta ^{n-1} \right)
%\end{align*}
%\end{proof}

\subsubsection{Example: a generalization of Welch functions}
\label{generalized Welch}

If $n$ is odd, then the function defined by 
\begin{align*}
f (x) = x^{2^t + 3},  \ \ \  t = \frac{n-1}{2} 
\end{align*}
is APN on $\mathbb{F}_{2^n}$ (see \cite{Dobbertin99Welch}). 
Such functions are called the Welch functions 
(see Table \ref{monomial APN}). 
Here we construct a generalization of Welch functions. 

\begin{proposition}
%Let $n$ be odd, and let 
Let 
\begin{align*}
d = p^t + p + 1, \ \ \ t = 
\left\{ 
\begin{array}{cl}
\frac{n-1}{2} & (\mbox{$n$ is odd}),  \\
\frac{n}{2} & (\mbox{$n$ is even}).   \\
\end{array}
\right.
\end{align*}
Then $f_d$ is a GAPN function of $\mathbb{F}_{p^n}$ 
if and only if 
$p = 2$ and $n$ is odd, or $p=3$. 
\end{proposition}

\begin{proof}
When $p=2$, 
the function $f_d$ is the Welch function 
if and only if $n$ is odd. 
%, and hence $f_d$ is a (generalized) 
%APN function on $\mathbb{F}_{2^n}$. 
Since $d^{\circ} (f_d) = 3$, 
by Proposition \ref{Prop: basic result}, (ii), the function $f_d$ is not a GAPN function on 
$\mathbb{F}_{p^n}$ when $p \geq 5$. 
Let $p=3$. We prove that $f_d$ is a GAPN function on $\mathbb{F}_{3^n}$. 
Since $d^{\circ} (f_d) = 3$, 
by Theorem \ref{Thm: criterion}, it is sufficient to show that
\begin{align} \label{GWelch}
D (\beta) = 1 + \beta + \beta^t \ne 0 \ \ 
\mbox{for any $\beta \in \overline{\mathbb{F}_p} \setminus \{ 1 \}$ with $\beta^n=1$}. 
\end{align}
Assume that $D (\beta) = 0$ for some $\beta \in \overline{\mathbb{F}_p} \setminus \{ 1 \}$ with $\beta^n=1$. 
If $n$ is even, then we have 
\begin{align*}
1 = \beta^n = \beta^{2t} = (-1-\beta)^2 = 
1 - \beta + \beta^2, \ \ 
\mbox{that, is, } \ \ \beta (\beta - 1) = 0, 
\end{align*}
which is absurd. 
If $n$ is odd, 
then we have 
\begin{align*}
1 = \beta^n = \beta^{2 t + 1} = \beta \left( -1 - \beta \right)^2 
= \beta - \beta ^2 + \beta ^3 , \ \ \mbox{that, is}, \ \ 
\beta^2 + 1 = \beta (\beta^2 + 1). 
\end{align*}
Since $\beta \ne 1$, we get $\beta^2 = -1$. Hence we have 
\begin{align*}
0 = 1+ \beta + \beta^t 
= \left\{ 
\begin{array}{c l}
\beta-1 & (t \equiv 0 \mod 4), \\
-\beta+1 & (t \equiv 1 \mod 4), \\
\beta & (t \equiv 2 \mod 4), \\
1 & (t \equiv 3 \mod 4). \\
\end{array}
\right.
\end{align*}
In any case, they are contradictions. 
Therefore we obtain (\ref{GWelch}). 
\end{proof}

\subsection{Monomial GAPN functions on $\mathbb{F}_{p^n}$ with maximum algebraic degree}
\label{maximum algebraic degree}

%In $F = \mathbb{F}_{p^n}$, we have that $x^{p^n} = x$, so we may assume that 
%$d < p^n$. 
%Then for fixed $n \in \mathbb{N}$, we have that 
%$
%d^{\circ} (f_d) \leq n (p-1)
%$. 
%When $d^{\circ} (f_d) = n (p-1)$ we have 
%\begin{align*}
%d = (p-1) \left( 1 + p + \cdots + p^{n-1} \right) 
%= p^n - 1. 
%\end{align*}
%Thus $f_d (x) = x^{p^n - 1} = 0$ or $1$. 
%Clearly, it is not a GAPN function on $F$. 
%Therefore monomial GAPN functions on $F$ are algebraic degree at most 
%$n (p-1) - 1$. 

%In \cite{GAPN1}, 
For odd prime $p$, the inverse permutation 
\begin{align*}
f_{p^n-2} \colon F \longrightarrow F, \ \ x \longmapsto x^{p^n-2} 
\end{align*}
is a GAPN function (see \cite[Section 3]{GAPN1}). 
Note that when $p=2$ the inverse function is APN if and only if $n$ is odd. 
Then $p^n-2$ has the $p$-adic expansion 
\begin{align*}
p^n - 2 = (p-1) \left( 1 + p + \cdots + p^{n-1} \right)- 1 
= (p-2)  + (p-1) \left( p + \cdots + p^{n-1} \right). 
\end{align*}
Hence $d^{\circ} (f_{p^n-2}) = n (p-1) - 1$. 
More generally, we have the following proposition: 

\begin{proposition}
Any monomial function on $F = \mathbb{F}_{p^n}$ with algebraic degree $n (p-1) - 1$ is 
EA-equivalent to the inverse permutation. 
%given by 
%\begin{align*}
%f_{(j)} \colon F \longrightarrow F, \ \ x \longmapsto x^{p^n - p^j -1} \ \ \ \ 
%\mbox{for some $j \in \left\{ 0, 1, \dots, n-1 \right\}$}. 
%\end{align*}
%Then these are GAPN functions on $F$. 
In particular, 
it is a GAPN function on $F$. 
%The monomial function defined by 
%\begin{align*}
%f_{(j)} \colon F \longrightarrow F, \ \ x \longmapsto x^{p^n - p^j -1} \ \ \ \ 
%\mbox{for some $j \in \left\{ 0, 1, \dots, n-1 \right\}$}
%\end{align*}
%is EA-equivalent to the inverse function $f_{p^n-2}$ ($ = f_{(0)}$). 
%In particular, $f_{(j)}$ is a GAPN function on $F$ of algebraic degree 
%$d^{\circ} \left( f_{(j)} \right) = n (p-1) - 1$. 
\end{proposition}
\begin{proof}
Since $w_p (d) = n (p-1) - 1$, the exponent $d$ is given by 
\begin{align*}
(p-1) (1 + p + \cdots + p^{n-1}) - p^j = p^n - p^j - 1, 
\end{align*}
for some $j \in \left\{ 0, 1, \dots, n-1 \right\}$. 
Hence any monomial function on $F$ with algebraic degree $n (p-1) - 1$ is given by 
\begin{align*}
f_{(j)} \colon 
F \longrightarrow F, \ \ x \longmapsto x^{p^n - p^j -1} %\ \ \ \ 
%\mbox{for some $j \in \left\{ 0, 1, \dots, n-1 \right\}$}
\end{align*}
for some $j \in \left\{ 0, 1, \dots, n-1 \right\}$. 
%We denote it by $f_{(j)}$. 
%For any $j \in \left\{ 0, 1, \dots, n-1 \right\}$, 
Let ${\rm Fb}_{(j)}$ be a Frobenius isomorphism 
\begin{align*}
{\rm Fb}_{(j)} \colon F \longrightarrow F, \ \ x \longmapsto x^{p^{n-j}}. 
\end{align*}
Then we have 
\begin{align*}
\left( f_{(j)} \circ {\rm Fb}_{(j)} \right) (x) = \left( x^{p^{n-j}} \right)^{p^n - p^j - 1} 
%= x^{p^{n-j}} \cdot x^{-1} \cdot x^{-(p^{n-j})} 
= x^{-1} 
= f_{p^n-2} (x), 
\end{align*}
where we put $0^{-1} := 0$. 
Hence $f_{(j)}$ is EA-equivalent to the inverse function $f_{p^n-2}$. 
By Proposition 2.1 in \cite{GAPN1}, $f_{(j)}$ is a GAPN function on $F$. 
%for any 
%$j \in \left\{ 0, 1, \dots, n-1 \right\}$. 
%Since 
%\begin{align*}
%p^n - p^j - 1 = (p-1) (1 + p + \cdots + p^{n-1}) - p^j, 
%\end{align*}
%we have that $d^{\circ} (f_{(j)}) = n (p-1) - 1$. 
\end{proof}

%\subsection{Monomial GAPN functions $f_d$ with $p < d^{\circ} (f_d) < n(p-1) - 1$}
\subsection{The other monomial GAPN functions on $\mathbb{F}_{p^n}$}

By simple computations, we can show that 
when $p=3$ and $n \in \left\{ 6, 7, 8 \right\}$, 
there are no monomial GAPN functions $f_d$ on $\mathbb{F}_{3^n}$ with $3 < d^{\circ} (f_d) < 2 n - 1$. 
More generally, we give the following conjecture: 

\begin{conjecture} \label{conj 1}
Let $p$ be an odd prime. 
For sufficiently large $n$, 
there are no monomial GAPN functions $f_d$ on $\mathbb{F}_{p^n}$ with $p < d^{\circ} (f_d) < n (p-1) - 1$. 
In particular, 
%any exceptional monomial GAPN function has algebraic degree $p$. 
%the converse of Proposition \ref{monomial exceptional GAPN} is true. 
the only $p$-exceptional exponents are given by (\ref{p-exceptional}). 
\end{conjecture}

\bibliographystyle{amsplain1}
\bibliography{bibfile}

\vspace{5mm}
\begin{table}[h]
%\centering
\begin{tabular}{l}
Masamichi Kuroda \\%& \hspace*{10mm}& Shuhei Tsujie \\
Department of Mathematics \\%& & Department of Mathematics \\
Hokkaido University \\%& & Hokkaido University \\
Sapporo 060-0810 \\%& & Sapporo 060-0810\\
Japan \\%& & Japan  \\
m-kuroda@math.sci.hokudai.ac.jp \\%& &  tsujie@math.sci.hokudai.ac.jp
\end{tabular}
\end{table}

\end{document}